\documentclass[12pt,twoside,letter]{amsart}
\usepackage{mathrsfs}
\usepackage{mathrsfs}
\usepackage{txfonts}
\usepackage{bbm}
\usepackage{mathrsfs}
\usepackage{amsfonts}
\usepackage{amsmath}
\usepackage{amssymb}
\usepackage{wasysym}
\usepackage{psfrag}
\usepackage{graphics,epsfig,amsmath}  % standard LaTex graphics tool
\usepackage{comment}
\usepackage{xcolor}
\usepackage{enumerate}
\newtheorem{theorem}{Theorem}

\newtheorem{Remark}{Remark}
\newtheorem{lemma}{Lemma}
\newtheorem{prop}{Proposition}

\def\eps{\varepsilon}
\def\complex{{\mathbb C}}
\def\torus{{\mathbb T}}
\def\real{{\mathbb R}}
\def\integer{{\mathbb Z}}

\def\dist{{\rm dist}}

\title[Resonant quasi-periodic equilibria]{Resonant Equilibrium configurations in quasi-periodic media: perturbative expansions}

\author[R. de la Llave]{Rafael de la Llave}
\address{School of Mathematics\\
Georgia Institute of Technology \\
686 Cherry St. \\
Atlanta GA 30332, USA}
\address{
JLU-GT joint institute for Theoretical Science\\
Jilin University\\
Changchun, 130012, CHINA}
\email{rafael.delallave@math.gatech.edu}

\author[X. Su] {Xifeng Su}
\address{School of Mathematical Sciences\\
Beijing Normal University\\
No. 19, XinJieKouWai St.,HaiDian District\\
 Beijing 100875, P. R. China}
\email{xfsu@bnu.edu.cn}

\author[L. Zhang]{Lei Zhang}
\address{School of Mathematics\\
Georgia Institute of Technology \\
686 Cherry St. \\
Atlanta GA 30332, USA}
\email{lzhang98@math.gatech.edu}

\thanks{L. Z and R. L. supported by NSF grant DMS-1500943}
\thanks{X. S supported by National Natural Science Foundation of China (Grant No. 11301513) and ``the Fundamental Research Funds for the Central Universities"}

\begin{document}
\maketitle
\vspace{0.1in}

\begin{abstract}
We consider 1-D quasi-periodic Frenkel-Kontorova models.

We study the existence of equilibria whose frequency (i.e.
the inverse of the density of deposited material) is resonant
with the frequencies  of the substratum.

We study perturbation theory for small potential. We show that there
are perturbative expansions to all orders for the quasi-periodic
equilibria with resonant frequencies.  Under very general conditions,
we show that there are at least two such perturbative expansions for equilibria
for small values of the parameter.

We also develop a dynamical interpretation of the equilibria in these quasi-periodic
media. We show that equilibria are orbits of a dynamical system which has very unusual properties.
We obtain results on the Lyapunov exponents of the dynamical systems, i.e. the phonon gap  of the resonant quasi-periodic equilibria. We show that the equilibria can be pinned even if the gap is zero.
\end{abstract}



\section{Introduction}
The goal of this paper is to formulate the theory of resonant equilibria in quasi-periodic Frenkel-Kontorova models. We argue that these equilibria play an important role in the phenomenon of pinning (the equilibria that survive after an external force is applied). In periodic Frenkel-Kontorova models, the role of resonances in pinning is well established.

We recall that in Frenkel-Kontorova models\cite{BK04, Selke2}, one considers configurations given by a sequence of real numbers (think of the position of a sequence of particles deposited on a 1-D material). The (formal) energy of the system is the sum of a term of interaction between nearest neighbors of the deposited material  and a term modeling interaction with the media. In the quasi-periodic Frenkel-Kontorova models studied here, the interacting potential will be a quasi-periodic function of the position reflecting that the medium is quasi-periodic. We will be interested in equilibria, i.e., configurations such that the derivatives of the (formal) energy with respect to the position of each of the particles vanish. We note that even if the energy is a formal sum, the equilibrium equations are well defined. More details of the models will be discussed in Section~\ref{models}.

In \cite{SuL1, SuL2}, one can find a rigorous mathematical theory of
quasi-periodic solutions whose frequency is not resonant (indeed
Diophantine) with the frequencies of the substratum. The rigorous
theory of \cite{SuL1, SuL2} also leads to efficient algorithms that
can compute these quasi-periodic solutions arbitrarily close to their
breakdown. Implementations of these algorithms and investigation of
the phenomena at breakdown appear in \cite{BlassL}. The paper  \cite{SuL1},
studies models with nearest neighbor interaction while
\cite{SuL2} studies the case of long range interactions.

Our motivation is to study the phenomena of ``depinning". When we add an external force (no matter how small) to the model, many quasi-periodic solutions disappear. Nevertheless, there are still other quasi-periodic solutions survive. This physically corresponds to the deposited material rearranges itself to withstand the force. It is well known in the periodic Frenkel-Kontorova models that the solutions which persist under forcing are resonant with the media. The papers \cite{SuL1, SuL2} also show that, in the quasi-periodic case, the smooth non-resonant solutions do not exist when there is an external force. Hence we are interested in resonant solutions in the quasi-periodic models.

%In this paper, we formulate precisely several questions motivated by the physics, investigate the answers using perturbative expansions and develop a dynamical interpretation of equilibria. We also provide specific forms of the depinning strength as a function of the parameters and use the dynamical interpretation to show that the solutions have zero phonon gap.

In this paper, we show that there exist at least two formal power series in the amplitude of the coupling describing resonant solutions even in the presence of external forces. The delicate analytic question of convergence of these series will be pursued in \cite{ZhangSL15}. We also develop a dynamical interpretation of the equilibria as orbits of a dynamical system. Using this dynamical interpretation, we show that the phonon gap of quasi-periodic equilibria is zero.

We hope that this paper can lay the ground work for future explorations. We will use rigorous mathematical tools to study the convergence of the formal series developed here in \cite{ZhangSL15}. We will also use numerical methods to explore in a non-rigorous but more quantative way some of the phenomena discussed here.

Equilibria in quasi-periodic media with a resonant frequency  have been investigated numerically in
\cite{vanErp'99, vanErp'01, vanErp'02}. These papers also studied the
phonon gap and found it to vanish when there are smooth solutions
(in agreement with the results here).

The variational and topological methods that have been proved useful in the periodic case do not extend to the quasi-periodic case in its full strength. Several interesting counterexamples are in \cite{LS'03, Federer'74}. Partial results are in \cite{Gambaudo, Aliste10, GPT13, KunzeO13}.

\section{Models considered and formulation of the problem}
\label{models}
We consider models of deposition in a quasi-periodic one-dimensional
medium.

If $x_n\in\mathbb{R}$ denotes the position of the $n$-th particle of the deposited material,
the state of the system is specified by a configuration, i.e. a sequence $\{x_n\}_{n\in\mathbb{Z}}$.
We associate the following formal energy to a configuration of
the system
\begin{equation}\label{formal energy}
\mathscr{S}(x) = \frac{1}{2}\sum_{n\in\mathbb{Z}} (x_{n+1} -  x_n - a )^2 -
V(x_n \alpha) - \lambda x_n
\end{equation} where $V:\mathbb{T}^d\rightarrow \mathbb{R}$ is an
analytic function, $\alpha\in \mathbb{R}^d$ is an irrational
vector and $a,\lambda$ are some real numbers.

A natural example often used is

\begin{equation}\label{formal energy:example}
\mathscr{S}(x) = \frac{1}{2}\sum_{n\in\mathbb{Z}} (x_{n+1} -  x_n - a )^2 +A\cos x_n+B\cos(\sqrt{2}x_n+\phi)- \lambda x_n
\end{equation} which corresponds to $\alpha=(1,\sqrt{2})$ and $V(\theta_1,\theta_2)=A\cos\theta_1+B\cos(\theta_2+\phi)$.

The term $( x_{n+1} - x_n -a )^2$ represents the interaction among
neighboring deposited atoms. The term $V(x_n\alpha)$ represents the
interaction with the substratum. The interaction at position $x\in\mathbb{R}$ is the quasi-periodic function $V(x\alpha)$. This models that the substratum is quasi-periodic.
The term $\lambda x_n$ has the interpretation of a constant field
applied to the model. In the case of deposited materials, we
can imagine that the sample is tilted and $\lambda$ is
the component of the gravity. The physical meaning of the equilibria corresponding to $\lambda\neq 0$ is that, when we apply a small external field, the configurations rearrange themselves so that they can respond to the force and do not slide. Of course, when the force is strong enough, no rearrangement is possible and all the configurations slide. This is the microscopic origin of static friction.

In the case of periodic media, the pinned solutions are known to correspond to resonant frequencies. Since the resonant tori do not survive, it is reasonable to consider resonant frequencies.

The existence of external forces  $\lambda$
is a very important novelty with respect to
the previous papers \cite{SuL1, SuL2}. It was shown in \cite{SuL1, SuL2} that if there is
non-resonant quasi-periodic solution, then
$\lambda = 0$. In our case, we will show how to construct quasi-periodic equilibria  with nontrivial $\lambda$ and will show how to compute perturbatively the range of such $\lambda$ for which solutions with a prescribed resonant frequency exist.

Without any loss of generality, we can assume that
\begin{equation}\label{alpha nonresonant}
k\cdot\alpha\notin \mathbb{N}\quad \forall ~k\in \mathbb{Z}^d-\{0\}.
\end{equation}
If there existed a resonance $k\cdot \alpha=0$, we could just use
less frequencies to express the quasi-periodic function.

\subsection{Equilibrium equations}
A configuration is in equilibrium if the forces acting on all the
particles vanish. Equivalently, the derivatives of the energy with
respect to the position of the particles vanish. That is,
\[
\frac{\partial \mathscr{S}}{\partial x_n}(x) = 0 \qquad \forall ~n\in
\mathbb{Z}.
\]

In the model \eqref{formal energy}, the equilibrium equations are
\begin{equation}\label{equilibrium}
x_{n+1} + x_{n-1} - 2 x_n + \partial_\alpha V(x_n\alpha) + \lambda = 0 \quad
\forall~ n\in \mathbb{Z}
\end{equation}
where $\partial_\alpha = \alpha\cdot \nabla$ and $\nabla$ is the usual gradient.

Note that even if the energy \eqref{formal energy} is just a formal
sum, the equilibrium equations \eqref{equilibrium} are well defined
equations.

It is very tempting to consider \eqref{equilibrium} as a dynamical system, so that we obtain $x_{n+1}$ as a function of $x_n$ and $x_{n-1}$. This system has very unusual properties. This will be pursued in Section \ref{sec:dynamical}.

\subsection{Quasi-periodic configurations, hull functions}
\label{sec:hull}
In this paper, we will be interested in quasi-periodic solutions of
frequency $\omega\in\mathbb{R}$.

These are configurations of the form
\begin{equation}\label{hull}
x_n = n\omega + h(n\omega\alpha),
\end{equation}where $h:\mathbb{T}^d\rightarrow \mathbb{R}$.

A configuration given by a hull function \eqref{hull} satisfies the
equilibrium equation \eqref{equilibrium} if and only if the hull
function $h$ satisfies
\begin{equation}\label{hull equilibrium}
h(n\omega \alpha + \omega\alpha) + h(n\omega \alpha - \omega\alpha)
- 2 h(n\omega \alpha) + \partial_\alpha V(n\omega\alpha + \alpha
h(n\omega\alpha)) +\lambda = 0.
\end{equation}

The equation \eqref{hull equilibrium} was considered in
\cite{SuL1,SuL2} when $\omega\alpha$ is Diophantine (in particular,
$n\omega\alpha$ is dense in the torus $\mathbb{T}^d$).

In our case, $n\omega \alpha$ will not be dense on the $d$-dimensional
torus (see Section~\ref{sec:resonances})
and the equilibrium equations we will derive are different from those in \cite{SuL1,SuL2}.

\subsection{Resonances}
\label{sec:resonances}
The goal of this paper is to study situations when there are
$k\in\mathbb{Z}^d-\{0\}$ and $m\in\mathbb{Z}$ such that
\begin{equation}\label{resonance}
k\cdot\omega\alpha - m = 0.
\end{equation}
When \eqref{resonance} holds we say that $(k,m)$ is a discrete
resonance for $\omega\alpha$ and we refer to the pair $(k,m)$ as a
resonance.

\begin{Remark}
Note that these discrete resonances \eqref{resonance}
are different from the resonances of the media
we excluded before ($k\cdot \alpha\neq 0$, $\forall ~k\in\mathbb{Z}^d-\{0\}$).
\end{Remark}
\begin{Remark}
If  \[k\cdot\alpha\neq0~\quad\forall
~k\in\mathbb{Z}^d\setminus \{0\},\]
given any $k_0 \in\mathbb{Z}^d \setminus \{0\}$, $m \in \mathbb{Z}$
we have that $\omega = -m/(k_0 \cdot \alpha)$ is a resonant frequency.
Since $k_0 \cdot \alpha$ can be arbitrarily large, we see that the
set of resonant frequencies is dense on the real line. Of course, once we fix $\alpha$, the set of resonant $\omega$ is a countable set.
\end{Remark}

\subsubsection{Multiplicity of a resonance}
 Clearly, if $(k, m)$, $(\tilde{k},
\tilde{m})$ are discrete resonances so is $(k+\tilde{k},
m+\tilde{m})$.

In mathematical language,
\[
\mathscr{M}_{\omega\alpha} = \left\{(k,m)\in
\mathbb{Z}^d\times\mathbb{Z}~:~k\cdot \omega\alpha-m =0 \right\}
\]
is a $\mathbb{Z}$-module called the resonance module for $\omega$.

We denote by $l(\omega) =dim(\mathscr{M}_{\omega\alpha})$ the dimension of
the resonance module and we call it the multiplicity of
the resonance.  The meaning of $l(\omega)$ is the number of independent
resonances. We can find $(k_1, m_1),\ldots, (k_l,m_l)$ in such a way
that all resonances can be expressed as combinations of the basic
resonances (and also no other set of basic resonances with smaller
number of elements will allow to express all the resonances).

\subsubsection{Only resonances of multiplicity 1 appear in the models \eqref{formal energy}}
\label{multiplicity 1}

In Hamiltonian mechanics for systems with $d$ degrees of
freedom,  one can find resonances of all
multiplicities up to $d$. As we will see later, in Section \ref{sec:dynamical}, one can give a dynamical interpretation of the equilibrium equations as a dynamical system in $d+1$ dimensions. Nevertheless, in our models only $l=1$ appears
independently of the number of degrees of freedom.
This highlights that the problem here is different from the Hamiltonian problem.
\begin{prop}
  If $\omega\alpha$ is resonant, i.e. $\mathscr{M}_{\omega\alpha}\neq \{0\}$, then $l(\omega)=1$.
\end{prop}

\begin{proof}
Note that
\[
k_1\cdot \omega\alpha - m_1 = k_2\cdot \omega\alpha - m_2 =0
\]
implies (because $m_1\neq 0, ~m_2\neq 0$ because of \eqref{alpha
nonresonant})
\[
\omega = \frac{m_1}{k_1\cdot \alpha} = \frac{m_2}{k_2\cdot \alpha}
\] and therefore
\[
\alpha\cdot (k_1 m_2 - k_2 m_1) = 0
\] and, because $\alpha$ is non-resonant \eqref{alpha nonresonant} we have
\[
k_1 m_2 = k_2 m_1.
\]
Therefore, the two resonant vectors are related.
\end{proof}

\subsubsection{The intrinsic frequencies}
When $\omega\alpha$ is resonant, we can find a matrix $B\in SL(d, \mathbb{Z})$,
$\Omega \in \mathbb{R}^{d-1}$, $L \in \mathbb{Z}^d$ in such a way that
\begin{equation}\label{Bdefined}
B\omega\alpha = (\Omega, 0)  + L \quad \text{with}\quad \Omega\cdot
k\notin \mathbb{Z} \text{ for } k\in
\mathbb{Z}^{d-1}-\{0\}.
\end{equation}

We will refer to $\Omega$'s as the intrinsic frequencies.
They are essentially unique, i.e., unique up to changes of basis in $\mathbb{R}^{d-1}$ given by a matrix in  $SL(d-1, \mathbb{Z})$.

In this case, the set $\{n\omega\alpha\}_{n\in\mathbb{Z}}$ has a
closure which is a $d-1$ dimensional torus. This torus is invariant
under the translation $T_{\omega\alpha}$. If we stay in this $d-1$ dimensional torus, $T_{\omega\alpha}$ can be described as $T_{\Omega}$. The torus $\mathbb{T}^d$ is
foliated by these $\mathbb{T}^{d-1}$ indexed by another parameter
$\eta\in\mathbb{T}^1$. We will write a point in $\mathbb{T}^d$ as
$(\psi,\eta)$ where $\psi$ is the coordinate corresponding to the
position in $\mathbb{T}^{d-1}$. The coordinate $\eta$ selects the
$d-1$ torus we are considering.

\subsection{Quasi-periodic equilibria
 with resonant frequencies}
\label{resonanthull}

The natural notion of the hull functions in the resonant case would be
to assume that the equilibrium solutions have the form
\begin{equation} \label{hull resonant}
x_n = n \omega + v(n \Omega)
\end{equation}
with $v:\mathbb{T}^{d-1} \rightarrow \mathbb{R}$.

Note that the physical meaning of $\omega$ is still the mean spacing
of the solutions (i.e.,  an inverse density). The term $v(n\Omega)$
represents fluctuations that can be parameterized in terms of the
intrinsic frequency $\Omega$. Of course, we could represent them in
terms of the original frequencies, but it is more natural to change
variables so that they become a part of the equation.

We will refer to the $\eta$ variable as the \emph{transversal phase}.
The resonant solutions considered here, cover densely a  torus
of codimension one.
The one-dimensional variable $\eta$ measures
the position of these codimension-one tori on the configuration space $\mathbb{T}^d$ corresponding to the internal phases of $V$.

If we substitute the parameterization \eqref{hull resonant}
into the equilibrium equation, we obtain that
the equilibrium equation \eqref{equilibrium} is equivalent to:
\begin{equation} \label{equilibrium resonant hull}
v(n \Omega + \Omega) +
v(n \Omega - \Omega)
- 2 v(n \Omega)  + \partial_\alpha V(n \omega \alpha  + \alpha v(n \Omega)) +\lambda = 0.
\end{equation}

If we furthermore introduce the notation
$\partial_\alpha V(\theta) =W(B \theta)$ and $B \alpha = \beta$,
and observe that the $n \Omega$
is dense on $\mathbb{T}^{d-1}$, we see that for continuous functions $v$, \eqref{equilibrium resonant hull}
is equivalent to:
\begin{equation}\label{equilibrium resonant hull3}
v(\psi + \Omega) +
v(\psi - \Omega)
- 2 v(\psi)  + W((\psi, 0)  + \beta v(\psi)) +\lambda = 0.
\end{equation}

Note that we can also consider solutions of the form
\begin{equation}\label{hull2}
  x_n=n\omega+v(n\Omega+\xi_1)+\xi_2
\end{equation}
for any fixed $\xi_1$ and $\xi_2$. This will give us freedom to add the transversal phase $\eta$ as an additional parameter in \eqref{equilibrium resonant hull3}. So the equilibrium equation we will consider is
\begin{equation}\label{equilibrium resonant hull2}
v(\psi + \Omega) +
v(\psi - \Omega)
- 2 v(\psi)  + W((\psi, \eta)  + \beta v(\psi)) +\lambda = 0.
\end{equation}

By simple calculations, it can be shown that \eqref{equilibrium resonant hull2} is equivalent to \eqref{equilibrium} when the hull function is of the form \eqref{hull2} and $B\alpha\xi_2=(\xi_1,\eta)$.

\begin{Remark}
Because $\beta$ has components
both in the $\psi$ and the $\eta$ directions, the equation
\eqref{equilibrium resonant hull2} cannot be considered as a parameterized version
of the equations considered in \cite{SuL1}.  As we will see,
the symmetries of the equation involve transformations that mix
the dependence in $\psi$ and in $\eta$.
\end{Remark}

\subsection{The symmetries of the invariance equation \eqref{equilibrium resonant hull2}}

The equation \eqref{equilibrium resonant hull2} possesses remarkable symmetries
that make the solutions not unique.  These symmetries lead to Ward identities.
In contrast with the case of non-resonant solutions, the group of symmetries
is infinite dimensional. In \cite{Rafael'08, SuL1, SuL2} these symmetries
are used to develop a KAM method.

The main observation is that
if $(v, \lambda)$ is a solution of \eqref{equilibrium resonant hull2},
then, for every $\iota(\eta): \mathbb{T}^1 \rightarrow  \mathbb{R}$,
the pair $(\tilde v,  \tilde\lambda)$ is also a solution of \eqref{equilibrium resonant hull2}
where we denote $\beta = (\beta_\psi, \beta_\eta)$ and  $\tilde v, \tilde \lambda$ are defined
by:
  \begin{equation}\label{symmetry}
  \begin{split}
  \tilde v(\psi,\eta) & =
  v\big((\psi,\eta) + \iota(\eta) \beta\big) + \iota(\eta),\\
 \tilde \lambda(\eta) &=  \lambda \big(\eta + \iota(\eta)\beta_\eta\big).
  \end{split}
  \end{equation}

Notice that the symmetry \eqref{symmetry} involves changing not only
the argument $\psi$ but also the argument $\eta$. Note the space of symmetries of the equation is not just a finite dimensional space but rather an infinite dimensional space of functions.

\subsection{A normalization of the solutions of the invariance equation \eqref{equilibrium resonant hull2}}
For later applications, it will be useful to have local
uniqueness of the solutions (e.g. to discuss smooth dependence on
parameters, perturbative expansions on parameters). Hence we impose the normalization
\begin{equation}\label{normalization}
\int_{\mathbb{T}^{d-1}} v(\psi,\eta) \, d \psi  = 0.
\end{equation}

Since the symmetry \eqref{symmetry} involves changes of
arguments,  giving a $v_\eta$, finding the $\iota(\eta)$ that accomplishes the normalization involves solving the implicit
equation
\begin{equation}\label{implicit}
 I(\eta + \beta_\eta \iota(\eta) ) + \iota(\eta)  = 0
\end{equation}
where $I(\eta) \equiv \int_{\mathbb{T}^{d-1} } v(\psi,\eta)\, d \psi$.

If $I$ and its derivative are small, one can solve eqref{implicit} using implicit function theorem.

\subsection{Diophantine condition}\label{Diophantine properties}
%We will assume that $\alpha \in\mathbb{R}^{d}$ is non-resonant. We are interested in the frequency $\omega\in \mathbb{R}$ such that \eqref{resonance} holds.

%Then, we can find a matrix $B\in SL(d, \mathbb{Z})$ as in \eqref{Bdefined}
%in such a way that $\Omega$ satisfies
%Diophantine condition in $\mathbb{R}^{d-1}$:
%\begin{equation}\label{Diophantine}
%|\hat k\cdot \Omega-m| \geq \nu |\hat{k}|^{-\tau} \qquad \forall~
%\hat k\in \mathbb{Z}^{d-1}-\{0\},~m\in \mathbb{Z}.
%\end{equation}
%Here $\nu,\tau$ are positive numbers and we denote such
%set of $\Omega$ by
%$\mathscr{D}(\nu,\tau)$.
%We also denote $\mathscr{D}(\tau) = \cup_{\nu > 0} \mathscr{D}(\nu,\tau)$.

In contrast with KAM theory, we will not need very delicate estimates on the solutions
and hence, we can deal with very general Diophantine conditions.
We will assume that $\Omega$ satisfies
\begin{equation} \label{subexponential}
\lim_{N\to\infty} \frac{1}{N} \sup_{|k| \le N, m \in \mathbb{Z}} \bigg| \ln |k\cdot \Omega-m| \bigg|  = 0.
\end{equation}
Note that the condition \eqref{subexponential} is much weaker than the usual
Diophantine conditions and even than the Bjruno-R\"ussmann conditions.
The condition\eqref{subexponential} is the natural condition in
the study of existence of series to all orders.

The following proposition shows that the sets of frequencies we are considering are abundant.

Fix a vector $k \in \integer^d \setminus \{0\}$, $m \in \integer\setminus\{0\}$ and assume without loss of generality that there is no common divisor in the components of $k$. For any $\alpha \in \real^d$ satisfying $\alpha\cdot k\neq 0$, we can find a unique $\omega$
such that $\alpha \cdot k \omega - m  = 0$. Fix a $B_k\in SL(d,\mathbb{Z})$ such that the last row of $B$ equals $k$ (It's possible to find such $B$ since the components of $k$ has no common divisor). Then, let $B_k\alpha\omega=(\Omega,0)+m$. Hence, for any $k, m$ and $\alpha$ satisfying $\alpha\cdot k\neq 0$, we can
define $\Omega$ as a function of $\alpha$. Denote $\Omega =  F_{k,m}(\alpha)$.

\begin{prop}
\label{prop:diophantine}
The set of $\alpha$ for which
$F_{k,m}(\alpha)$ satisfies \eqref{subexponential} for all $k,m$ is of full measure in $\mathbb{R}^d$. In other words, for a full measure set of medium frequencies, we can find a countable many resonant frequencies that lead to intrinsic frequencies satisfying \eqref{subexponential}.
\end{prop}

\begin{proof}
Since countable intersections of sets of full measure are of full measure, to prove Proposition~\ref{prop:diophantine} it suffices to show that for a fixed $k,m$ as above, the set $\{\alpha \in \real^d|\alpha\cdot k\neq 0\ \mathrm{and}\ F_{k,m}(\alpha)\ \mathrm{satisfies}\ \eqref{subexponential}\}$ is of full measure.

Because the set of $\Omega$'s which satisfy \eqref{subexponential} is of full measure on $\real^{d -1}$ and the linear map $B_k$ is differentiable and surjective, the preimage of the set of $\Omega$'s that satisfy \eqref{subexponential} under $B_k$ is also of full measure in the hyperplane $\Gamma=\{\gamma\mid \gamma\cdot k-m=0\}$. Denote the set of preimages as $\Gamma'$. Then any nonzero scaling of an element of $\Gamma'$ will give an $\alpha$ we want, which also form a full measure set in $\mathbb{R}^d$.
\end{proof}

\section{Function spaces and linear estimates}\label{cohomology}
The main tool that we will use  to construct perturbation theories is
the solution of cohomology equations.

We denote
\[
D_\rho = \{ \theta \in \complex^d/\integer^d  \mid | \text{Im}(\theta_i)|  < \rho\}
\]
and denote the Fourier expansion of a periodic mapping $v(\psi,\eta)$ on
$D_\rho$ by
\begin{equation*}
v(\psi,\eta)=\sum_{k\in \mathbb{Z}^d} v_k e^{2\pi i k\cdot (\psi,\eta)},
\end{equation*} where $\cdot$ is the Euclidean scalar product in
$\mathbb{C}^d$ and $v_k$ are the Fourier coefficients.

We denote by $\mathscr{A}_\rho$ the Banach space of analytic
functions on $D_\rho$ which are real for real argument and extend
continuously to $\overline{D_\rho}$. We make $\mathscr{A}_\rho$ a
Banach space by endowing it with the supremum norm:
\begin{equation*}
\|v\|_\rho =\sup_{(\psi,\eta)\in \overline{D_\rho}}|v(\psi,\eta)|.
\end{equation*}

These Banach spaces of analytic functions
are the same spaces as in \cite{Moser'67}.

We will consider equations of the form
\begin{equation}\label{Cohomology equation}
v(\psi+\Omega,\eta)-v(\psi,\eta)=\phi(\psi,\eta),
\end{equation}
where $\psi \in \torus^{d-1}$.

 To simplify our notations, we  will denote $v(\psi+\Omega)$ and $v(\psi-\Omega)$ as $v_+$ and $v_-$, respectively. Similar notations will be used for other functions. We also use $T$ to represent the translation operators, i.e., $T_{\Omega}v(\psi)=v(\psi+\Omega)$.

\begin{lemma}\label{estimate lemma for cohomology equation}
Let $\phi \in\mathscr{A}_\rho(\mathbb{T}^d)$  be such that
\begin{equation}\label{normalization phi}
\int_{\mathbb{T}^{d-1}}\phi (\psi,\eta)d\psi=0,
\end{equation}
for all $\eta$.

Assume that $\Omega$ satisfies the assumption
\eqref{subexponential}.

Then, for a fixed $\eta$, there
exists a unique solution $v_\eta$ of
\eqref{Cohomology equation} which satisfies
\begin{equation}\label{normalization equation}
\int_{\mathbb{T}^{d-1}} v (\psi,\eta)d\psi = 0.
\end{equation}
The solution $v \in \mathscr{A}_{\rho'} $ for
any $\rho' < \rho$
and we have
\[
|| v ||_{\rho'} \le C(\rho,\rho') || \phi_\eta||_\rho.
\]

Furthermore, any distribution solution of \eqref{Cohomology
equation} differs from the solution claimed before by a constant.

If $\phi$ is such that it takes real values for real arguments,
so does $v$.

If we consider now the dependence in $\eta$, we
have that
$v \in \mathscr{A}_{\rho'}(\mathbb{T}^d)$
and
\[
||v||_{\rho'} \le C(\rho, \rho') ||\phi||_\rho.
\]
\end{lemma}

\begin{proof}
We note that, as it is well known that obtaining $v$
solving \eqref{Cohomology equation} for
given $\phi$  is very explicit in
terms of Fourier coefficients.
If
\[
\phi(\psi,\eta) =  \sum_{k\ne 0}\hat \phi_k(\eta) e^{2 \pi i k \cdot \psi}
= \sum_{k \ne 0 , m} \hat \phi_{k,m}  e^{2 \pi i (k \cdot \psi + m \eta)}
\]
then,  $v$ is given by
\[
v(\psi,\eta) =
 \sum_{k \ne 0} \hat \phi_k(\eta)(e^{2 \pi i k \cdot \Omega} -1)^{-1}
e^{2 \pi i k \cdot \psi}
= \sum_{k\ne 0, m} \hat \phi_{k,m} (e^{2 \pi i k \cdot \Omega} -1)^{-1} e^{2 \pi i (k \cdot \psi + m \eta)}.
\]
Using Cauchy estimates for the Fourier coefficients
$|\hat \phi_{k,m} |  \le  \exp( -2 \pi \rho ( |k| + |m|)) ||\phi||_\rho$
and that
$| e^{2 \pi k \cdot \Omega} -1|^{-1} \le C \dist(k \cdot \Omega, \integer)^{-1} $
and the assumption~\eqref{subexponential}, we obtain
that
\[
\begin{split}
||v||_{\rho'}
& \le C \sum_{k\ne 0, m} \exp( -2 \pi \rho ( |k| + |m|)) ||\phi||_\rho
\dist(k \cdot \Omega, \integer)^{-1}
 ||e^{2 \pi i (k\cdot \psi + m \eta)} ||_{\rho'} \\
& \le C ||\phi||_\rho\sum_{k\ne 0, m} \exp( -2 \pi \rho ( |k| + |m|))
{\rm dist}(k \cdot \Omega, \integer)^{-1}
\exp( 2 \pi \rho' ( |k| + |m|))\\
& \le C(\rho,\rho')||\phi||_{\rho}.
\end{split}
\]
\end{proof}

\section{Lindstedt series for quasi-periodic solutions with resonant frequencies}
\label{Lindstedt}

The goal of this section is to study \eqref{equilibrium resonant hull2}
perturbatively when the non-linear term  is small. Hence, we will
write \eqref{equilibrium resonant hull2} with a small parameter $\eps$
\begin{equation}\label{equilibrium revised}
\begin{split}
v(\psi +\Omega,\eta) + v(\psi -\Omega,\eta) -
2 v(\psi,\eta) + \epsilon  W((\psi, \eta)  + \beta v(\psi,\eta)) + \lambda(\eta)&=0 .\\
\end{split}
\end{equation}

We will find $v(\psi,\eta), \lambda(\eta)$
solving  \eqref{equilibrium revised} and \eqref{normalization} in the sense of
formal power
series in $\epsilon$.
In this paper, we will not consider the problem of whether these series converge or represent a function. This will be studied in more details in \cite{ZhangSL15}.

Since one may want to find solutions correspond to $\lambda=0$ (or $\lambda=\lambda^*$ with $|\lambda^*|$ small), it is important for us to keep track of
$\frac{\partial \lambda}{\partial \eta}(\eta,\epsilon)$ in order to solve $\lambda(\eta,\epsilon)=0$ by implicit function theorem.

 Following the standard perturbative procedure we will write
 \begin{equation}
\label{seriesexpansion}
\begin{split}
  v & = \sum_{n= 0}^\infty \epsilon^n v^n,\\
\lambda &= \sum_{n= 0}^\infty \epsilon^n \lambda^n.
\end{split}
\end{equation}

Here $v^n$ and $\lambda^n$ are coefficients of $\epsilon^n$, not powers of $v$ or $\lambda$. Substitute \eqref{seriesexpansion} in \eqref{equilibrium revised}
and equate powers of $\epsilon$.

Of course, carrying out this procedure for $n\leq N$ will require
that $\Omega$ satisfies some Diophantine properties as well as some differentiability
assumptions.

Equating the coefficients of $\epsilon^0$ in \eqref{equilibrium revised}
  we obtain
  \begin{equation}
 \begin{split}
  v ^ 0(\psi+ \Omega,\eta) + v^0(\psi - \Omega,\eta) -
 2v^0(\psi,\eta) + \lambda^0(\eta) &=0.\\
\end{split}
\end{equation}

 Hence, if $\Omega$ satisfies the condition \eqref{subexponential}  we see that $v^0$ is constant,
  $\lambda^0 = 0$ and imposing the normalization \eqref{normalization} we obtain
  $v^0 = 0$.

 Matching coefficients of $\epsilon^1$ in both sides of  \eqref{equilibrium revised}
 we obtain
  \begin{equation}\label{order1}
\begin{split}
  v^1(\psi+\Omega,\eta) + v^1(\psi-\Omega,\eta) -
 2v^1(\psi,\eta) + W (\psi, \eta)    + \lambda^1(\eta) &= 0.
 \end{split}
  \end{equation}
We see that, using the theory in Section ~\ref{cohomology}, to have analytic
 $v^1$   solving \eqref{order1}, it is necessary and sufficient to
  have
 \begin{equation}\label{lambda1}
 \begin{split}
 \lambda^1(\eta) &= -\int_{\mathbb{T}^{d-1}} W (\psi, \eta) d\psi.
 \end{split}
  \end{equation}

 Then, $v^1$, $\lambda^1$  can be determined uniquely up to a constant from
\eqref{order1}.  In fact, in Fourier series, the equation for $v^1, \lambda^1 $
is
\begin{equation}
 \begin{split}
v_k ^1 2 (\cos (2\pi k \Omega) -1) & = - W_k - \delta_{0, k} \lambda^1,\\
\end{split}
 \end{equation}
where $\delta_{0,k}$ is the Kronecker delta.
 In particular, the constant in $v^1$ is determined by the normalization \eqref{normalization}.

Proceeding to higher order follows the same pattern. We see that
 matching the terms of order $\epsilon^n$ in \eqref{equilibrium revised} we obtain
  \begin{equation}\label{ordern}
\begin{split}
  v^n (\psi +\Omega,\eta) + v^n (\psi -\Omega,\eta) -
 2v^n (\psi,\eta) + R^n(\psi,\eta) + \lambda^n(\eta) &=0,\\
\end{split}
\end{equation}where $R^n$ is a polynomial expression in $v^1,\ldots,
  v^{n-1}$ with coefficients which are derivatives with respect
  to $\psi$ of $W((\psi, \eta)  + \beta v(\psi,\eta))$.
This polynomial can be computed
explicitly because it is given by
\begin{equation}\label{R^n}
R^n = \frac{1}{(n-1)!}\frac{d^{n-1}}{d\epsilon^{n-1}} W
\bigg((\psi,\eta) + \beta \sum_{j=0}^{n-1} v^j
(\psi,\eta)\bigg)\bigg|_{\epsilon=0}
\end{equation} and these are well known formulae. We also note that,
from the algorithmic point of view there are efficient ways to
compute $R^N$ using methods of ``automatic
differentiation'' \cite{Harodifferentiation, BCHNN06}.

Since $R^n$ can be computed explicitly, \eqref{ordern} can be solved the same way as \eqref{order1}.

We have therefore established
\begin{theorem}\label{LindstedtN}
Assume that $\Omega$ satisfies \eqref{subexponential}  and that $W:D_{\rho} \rightarrow \mathbb{C}$
 is an analytic function. We can find formal power series solutions in $\epsilon$ of the form
\eqref{seriesexpansion} solving the equation \eqref{equilibrium revised}. For any $0<\delta<\rho$, each of the terms $v^n(\psi,\eta)$ is analytic in $D_{\rho-\delta}$. If $W$ takes real values for real values,
 then so do $v^n$, and $\lambda$ is real.
\end{theorem}

\subsection{The auxiliary equation}
Now, we turn to the problem of studying the equation
\begin{equation}\label{auxiliary}
\lambda(\eta, \epsilon) = \lambda^*.
\end{equation}
We expect to obtain a solution $\eta^\ast(\epsilon)$ provided that \eqref{auxiliary} satisfies some non-degeneracy conditions.

Having solution of \eqref{auxiliary} to order 1 in $\epsilon$, amounts to
\[
\lambda^1(\eta) =0.
\]
That is, we need to find $\eta$ such that
\begin{equation}\label{auxiliaryfirst}
\int_{\mathbb{T}^{d-1}} W((\psi, \eta)) d\psi =0.
\end{equation}

\begin{theorem}
\label{prop:twosolutions}
The equation \eqref{auxiliaryfirst} has always two solutions.
\end{theorem}
\begin{proof}
Since
\[
\int_{\mathbb{T}^{d-1}} W((\psi, \eta)) d\psi =
\int_{\mathbb{T}^{d-1}} (\partial_\alpha V)\big(B^{-1}(\psi,\eta) \big)
d\psi,
\]if we integrate again with respect to $\eta$ we obtain
\begin{equation}
\int_{\mathbb{T}}\int_{\mathbb{T}^{d-1}} W((\psi, \eta) d\psi\ d\eta = \int_{\mathbb{T}^{d}} (\partial_\alpha V)\big(B^{-1}(\psi,\eta)\big)
d\psi\ d\eta = 0.
\end{equation}

Hence the function of $\eta$ given by $\int_{\mathbb{T}^{d-1}}
W((\psi, \eta) d\psi$ is a continuous periodic function of $\eta$
with zero average. Therefore, it has at least two zeros. We also note that there are open sets of perturbations where there are 4,6,$\cdots$ zeros.
\end{proof}

Denote one of these solutions of \eqref{auxiliaryfirst} as
$\eta^\ast$.

A sufficient condition that ensures that we can solve the equation
\eqref{auxiliary} to all orders is that
\begin{equation}\label{nondeg}
\frac{\partial}{\partial \eta} \lambda^1(\eta, \epsilon)
\left|_{\eta=\eta^\ast,\epsilon=0}\right. \neq 0.
\end{equation}

More explicitly,
\begin{equation}\label{nondegeneracy}
\int_{\mathbb{T}^d} \frac{\partial}{\partial \eta} (\partial_\alpha V)\big(B^{-1}(\psi,\eta)\big)
d\psi\ d\eta \neq 0.
\end{equation}
Then, the implicit function theorem for power series
\cite{Cartan95,Dieudonne71} gives us that we can indeed find $\eta^\ast(\epsilon)$.

Similarly, we can solve the equation $\lambda(\eta) = \lambda^*$ provided that $|\lambda^*|$ is
sufficiently small.

Therefore, we have established
\begin{theorem}\label{secondLindstedt}
Assume that $\Omega\in \mathscr{D}(\nu,\tau)$ as defined in
\eqref{subexponential} , that $W$ is an analytic function,
and that \eqref{nondegeneracy} holds,
 we can find formal power series $\eta_\epsilon$ in $\epsilon$
so that $v_{\eta_\epsilon}$ is the solution of \eqref{equilibrium revised}.
\end{theorem}

 Clearly, since the function $\lambda^n(\eta)$ are bounded, if
$\lambda^*$ -- the physical force -- is large  enough, there is no solution.
This has a clear physical meaning. If we increase the external force but keep it small,
the system can react by changing the transversal phase. If the force increases beyond
a threshold, the system cannot react by adapting the phase. Hence, the equilibrium breaks
down.  In this paper, we are not considering the dynamics of
the model, only the equilibria (our models
for the energy include only the potential energy of
the configuration and not any kinetic energy).
One can, however, expect that, if there was some dynamics,
the equilibria considered here could slide.

Of course, the sufficient condition  \eqref{nondegeneracy} is
far from being necessary and there are many other conditions that
are enough.

\begin{prop}\label{secondLindstedtdegenerate}
Assume that $\Omega$ satisfies \eqref{subexponential} , that $W((\psi, \eta)  + \beta v(\psi,\eta)) $ is an analytic function,
and that \eqref{nondegeneracy} holds.

Assume that $\eta^*$ is such that for some $m \in \mathbb{N}$
we have
\begin{equation}
\begin{split}
& \lambda^i(\eta^*) = 0, \quad i = 1,\ldots 2 m\\
& \lambda^{2 m +1} (\eta^*) \ne 0.
\end{split}
\end{equation}

Then, we can find formal power series $\eta_\epsilon$ in $\epsilon$
so that $v(\psi,\eta_\epsilon)$ is the solution of \eqref{equilibrium revised}.
\end{prop}

The proof is again an application of the implicit function theorem for power series.

\section{A dynamical interpretation of the equilibrium equations of Frenkel-Kontorova models}
\label{sec:dynamical}

In this section, we present a dynamical interpretation of the
equilibrium equations \eqref{equilibrium} in Frenkel-Kontorova models.

Even if the dynamical interpretation is  possible for finite
range interactions, we see that adding another small interaction of
longer range is a singular perturbation (even the dimension of the phase space changes). Whereas, for the methods in
this paper, adding a small term in the longer range is a regular perturbation of the same order.

A straightforward way of transforming the equilibrium equation
\begin{equation}
x_{n+1} + x_{n-1} - 2 x_n + \epsilon\partial_\alpha V(x_n\alpha) + \lambda = 0 \quad
\forall~ n\in \mathbb{Z}
\end{equation}
into a
dynamical system is setting
\begin{equation}\label{evolution}
\begin{split}
y_n&=(x_n, x_{n-1})\\
y_{n+1}&= (2y_n^1- y_n^2 - \epsilon\partial_\alpha V(\alpha y_n^1)-\lambda, y_n^1).
\end{split}
\end{equation}

However, \eqref{evolution} is not very useful because we have to
consider it as a map of $\mathbb{R}^2$ and the term $\partial_\alpha V
(\alpha y_n ^1)$ does not make apparent that it is periodic in $\alpha
y_n^1$.

A more natural formulation is obtained by observing that the equation
\eqref{equilibrium} is equivalent to the system on $\mathbb{T}^d\times
\mathbb{R}$
\begin{equation}\label{standard}
\begin{split}
p_{n+1} &= p_n - \epsilon\partial_\alpha V(q_n) - \lambda\\
q_{n+1} &= q_n + \alpha p_{n+1},
\end{split}
\end{equation}where $q_n\in\mathbb{T}^d$, $p_n\in \mathbb{R}$.
(Just multiply \eqref{equilibrium} by $\alpha$ and use the substitution $p_n=x_n-x_{n-1}$, $q_n=\alpha x_n$. Note that \eqref{equilibrium} is equivalent to
\[
(x_{n+1} - x_n) - (x_n - x_{n-1}) + \epsilon\partial_\alpha V(\alpha x_n) + \lambda =0
\]
hence, we obtain the first equation.)

We will write the mapping~\eqref{standard} as
\begin{equation}\label{standardmap}
(p_{n+1}, q_{n+1}) = F_{\eps,\lambda}( p_n, q_n).
\end{equation}

Note that \eqref{standard} is
typographically  very similar to the standard map \cite{Chirikov}
or to analogues introduced for volume preserving maps.
Nevertheless, there are significant differences
(besides the different dimensions).

A very crucial  difference between \eqref{standardmap}
and the generic volume preserving maps is that
$q_{n+1} - q_n$ is always a multiple of $\alpha$ (see \eqref{standard}).
So that the two dimensional leaves
\begin{equation} \label{constraints}
\mathcal{M}_{q_0} = \{ (p, q_o  + \alpha t) ~ |~ p, t \in \mathbb{R}\}
\end{equation}
are preserved.
Note that each of the leaves $\mathcal{M}_{q_0}$ is dense in the $d+1$
dimensional phase space.

The mapping \eqref{standard} clearly preserves the volume form
$dp\wedge dq_1 \wedge \ldots \wedge q_d$ since it is the composition of
\begin{equation}\label{map1}
\begin{split}
p_{n+1} &= p_n - \epsilon\partial_\alpha V(q_n) - \lambda\\
q_{n+1} &= q_n
\end{split}
\end{equation}
and
\begin{equation}\label{map2}
\begin{split}
p_{n+1} &= p_n \\
q_{n+1} &= q_n + \alpha p_{n+1}.
\end{split}
\end{equation}

We recall that, in our context,
a volume preserving map
is exact when
$F^* (p dq_1 \wedge  d q_2 \wedge \ldots \wedge d q_d) =
p dq_1 \wedge  d q_2 \wedge \ldots \wedge d q_d  + d P
$ where $P$ is $d-1$ form.

Indeed, \eqref{standardmap} is an exact volume preserving map  if and only if
$\lambda = 0$, since it is easy to observe that, when $\lambda = 0$,   both \eqref{map1}
 and \eqref{map2} are exact.

When $\eps = 0$, $\lambda = 0$, the map \eqref{standard}
is integrable. The codimension-one tori given by $p = \text{cte.}$
are invariant and the motion in them is a rotation.

The volume preserving KAM theory leads us to expect that for $\epsilon\ll 1,\lambda=0$, the tori in which the frequency of the motion is Diophantine survive. We also expect that the tori with resonance, breaks down into lower dimensional tori. The lower dimensional tori can survive for $|\lambda|\ll 1$ (depending on $\epsilon$).

In this paper we have quantitative (but formal and non-rigorous) prediction of these phenomenon based on perturbative expansions. We hope that some of them may be either verified by rigorous results or explored numerically.

\subsection{On the global geometry of the constraints given
by \eqref{constraints} }

Integrable systems with constraints have been studied extensively
in geometric mechanics. Nevertheless, the systems we consider here
have some unusual properties that we would like to highlight.

It is customary to classify the constraints in holonomic when the
distributions are integrable (in the sense that they foliate the phase
space with a smooth quotient) and non-holonomic when the distributions
are not integrable  and they violate the
hypothesis of Frobenius Theorem \cite{Souriau,Audin,Holm1}.

The constraints \eqref{constraints} escape this dichotomy.
They are locally integrable (they do satisfy the hypothesis of
Frobenius Theorem and are locally given by invariant manifolds that give rise
to a foliation) but nevertheless, the manifolds  are dense, so that
they do not give
a nice quotient manifold.

Hence, even if we have holonomic constraints locally (and the infinitesimal results
about holonomic systems are applicable), some global aspects such as symplectic reduction \cite{Meyer73,
 MarsdenW74,MarsdenW01} cannot be applied to \eqref{standard}.

\subsection{Lyapunov exponents and phonon localization}\label{sec:phonon}
In this section we study the so called \emph{phonon gap} around the equilibria of \eqref{equilibrium resonant hull}
given by a hull function.

Let us start by recalling some standard definitions.
The main idea is that sound waves are defined by the propagation of
infinitesimal disturbances around an equilibrium equation.

If we linearize around an equilibrium solution $x=\{x_n\}_{n\in\mathbb{Z}}$, we
obtain the dynamics of the infinitesimal perturbations
$\xi_n$ is given by
\begin{equation} \label{evolutionlinear}
\ddot \xi_n = \xi_{n+1} + \xi_{n-1} - 2 \xi_n  + (\partial_\alpha)^2 V(\alpha x_n)
\xi_n  \equiv (\mathcal{L}_{x}  \xi)_n.
\end{equation}

It is clear that the propagation properties of sound waves will
be affected by the spectral properties of the operator $\mathcal{L}_x$.

Note that the operator $\mathcal{L}$ is a one-dimensional
Schr\"odinger operator with a position dependent potential.
The dependence will be given by the dynamics of the $x_n$. In particular, for
the solutions given by a hull function, we will be considering quasi-periodic
potentials.

The mathematical theory of the spectrum of
quasi-periodic Schr\"odinger operators is well developed
\cite{PasturF92,HaroL}. In particular, it is known that the spectrum is
independent of the $\ell^p$ space in which it is considered,
and, more important for us, that the spectrum can be characterized
by the existence of approximate eigenfunctions. In the
dynamical interpretation in  this section,
the spectrum corresponds to the Lyapunov exponents of the
solution \cite{AubryKB92}.

In the case of \eqref{standard}, we can study the Lyapunov spectra for any orbit using the
geometric constraints \eqref{constraints}.

\begin{prop}
Let $(p_n,q_n)$ be an orbit of the mapping given by \eqref{standard}. Assume that it is an orbit in the full measure set that Osledets
Theorem applies.  Then, $d-1$ Lyapunov exponents are
zero. Also, the sum of all the Lyapunov exponents is zero.
\end{prop}

\begin{proof}
Consider $\tilde F$, the lift of the map $F$ in \eqref{standardmap}.

Let $s$ be a vector perpendicular to $\alpha$. It is a simple computation
to show that:
\[
\tilde F( \tilde M_{q_0 + s} ) = \tilde M_{q_0} + s.
\]

Then it is clear that the $d-1$ vectors in the directions perpendicular to $s$ do
not grow.

The fact that the sum of the Lyapunov exponents for orbits of
a volume preserving map is
zero is  well known since the sums of the Lyapunov exponents is
the rate of growth of the determinant of iterates of the map.
\end{proof}

Of course, the dynamical system \eqref{standard} is straightforward to implement numerically and allows study of statistical properties of depinning.

\section*{Acknowledgements}
We thank Dr. T. Blass for discussions.
R. L. and L. Z.  have  been supported by DMS-1500943.
The hospitality of
JLU-GT Joint institute for Theoretical Sciences for the three authors
was instrumental in finishing the work.
 R.L also acknowledges the hospitality of
the Chinese Acad. of Sciences.
X. Su is supported by both National
Natural Science Foundation of China (Grant No. 11301513) and ``the Fundamental Research Funds for the Central Universities".

\bibliographystyle{alpha}
\bibliography{reference}

\newcommand{\etalchar}[1]{$^{#1}$}
\def\cprime{$'$}
\begin{thebibliography}{vEFRJ99}

\bibitem[AMB92]{AubryKB92}
S.~Aubry, R.~S. MacKay, and C.~Baesens.
\newblock Equivalence of uniform hyperbolicity for symplectic twist maps and
  phonon gap for {F}renkel-{K}ontorova models.
\newblock {\em Phys. D}, 56(2-3):123--134, 1992.

\bibitem[AP10]{Aliste10}
Jos{\'e} Aliste-Prieto.
\newblock Translation numbers for a class of maps on the dynamical systems
  arising from quasicrystals in the real line.
\newblock {\em Ergodic Theory Dynam. Systems}, 30(2):565--594, 2010.

\bibitem[Aud08]{Audin}
Mich{\`e}le Audin.
\newblock {\em Hamiltonian systems and their integrability}, volume~15 of {\em
  SMF/AMS Texts and Monographs}.
\newblock American Mathematical Society, Providence, RI; Soci\'et\'e
  Math\'ematique de France, Paris, 2008.
\newblock Translated from the 2001 French original by Anna Pierrehumbert,
  Translation edited by Donald Babbitt.

\bibitem[BCH{\etalchar{+}}06]{BCHNN06}
Martin B{\"u}cker, George Corliss, Paul Hovland, Uwe Naumann, and Boyana
  Norris, editors.
\newblock {\em Automatic differentiation: applications, theory, and
  implementations}, volume~50 of {\em Lecture Notes in Computational Science
  and Engineering}.
\newblock Springer-Verlag, Berlin, 2006.
\newblock Papers from the 4th International Conference on Automatic
  Differentiation held in Chicago, IL, July 20--24, 2004.

\bibitem[BdlL13]{BlassL}
Timothy Blass and Rafael de~la Llave.
\newblock The analyticity breakdown for {F}renkel-{K}ontorova models in
  quasi-periodic media: numerical explorations.
\newblock {\em J. Stat. Phys.}, 150(6):1183--1200, 2013.

\bibitem[BK04]{BK04}
O.~M. Braun and Y.~S. Kivshar.
\newblock {\em The {F}renkel-{K}ontorova model}.
\newblock Texts and Monographs in Physics. Springer-Verlag, Berlin, 2004.
\newblock Concepts, methods, and applications.

\bibitem[Car95]{Cartan95}
Henri Cartan.
\newblock {\em Elementary theory of analytic functions of one or several
  complex variables}.
\newblock Dover Publications Inc., New York, 1995.
\newblock Translated from the French, Reprint of the 1973 edition.

\bibitem[Chi79]{Chirikov}
Boris~V. Chirikov.
\newblock A universal instability of many-dimensional oscillator systems.
\newblock {\em Phys. Rep.}, 52(5):264--379, 1979.

\bibitem[Die71]{Dieudonne71}
Jean Dieudonn{\'e}.
\newblock {\em Infinitesimal calculus}.
\newblock Hermann, Paris, 1971.
\newblock Translated from the French.

\bibitem[dlL08]{Rafael'08}
Rafael de~la Llave.
\newblock K{AM} theory for equilibrium states in 1-{D} statistical mechanics
  models.
\newblock {\em Ann. Henri Poincar\'e}, 9(5):835--880, 2008.

\bibitem[dlLH10]{HaroL}
R.~de~la Llave and A~Haro.
\newblock Spectral theory and dynamics.
\newblock 2010.
\newblock Manuscript.

\bibitem[dlLSZ15]{ZhangSL15}
Rafael de~la Llave, Xifeng Su, and Lei Zhang.
\newblock Resonant equilibrium configurations in quasi-periodic media: {KAM}
  theory.
\newblock {\em preprint}, 2015.

\bibitem[Fed75]{Federer'74}
Herbert Federer.
\newblock Real flat chains, cochains and variational problems.
\newblock {\em Indiana Univ. Math. J.}, 24:351--407, 1974/75.

\bibitem[GGP06]{Gambaudo}
Jean-Marc Gambaudo, Pierre Guiraud, and Samuel Petite.
\newblock Minimal configurations for the {F}renkel-{K}ontorova model on a
  quasicrystal.
\newblock {\em Comm. Math. Phys.}, 265(1):165--188, 2006.

\bibitem[GPT13]{GPT13}
Eduardo Garibaldi, Samuel Petite, and Philippe Thieullen.
\newblock Discrete weak-kam methods for stationary uniquely ergodic setting.
\newblock {\em preprint}, 2013.

\bibitem[Har11]{Harodifferentiation}
Alex Haro.
\newblock Automatic differentiation methods in computational dynamical systems:
  invariant manifolds and normal forms.
\newblock 2011.

\bibitem[Hol11]{Holm1}
Darryl~D. Holm.
\newblock {\em Geometric mechanics. {P}art {I}}.
\newblock Imperial College Press, London, second edition, 2011.
\newblock Dynamics and symmetry.

\bibitem[KO13]{KunzeO13}
Markus Kunze and Rafael Ortega.
\newblock Twist mappings with non-periodic angles.
\newblock 2065:265--300, 2013.

\bibitem[LS03]{LS'03}
Pierre-Louis Lions and Panagiotis~E. Souganidis.
\newblock Correctors for the homogenization of {H}amilton-{J}acobi equations in
  the stationary ergodic setting.
\newblock {\em Comm. Pure Appl. Math.}, 56(10):1501--1524, 2003.

\bibitem[Mey73]{Meyer73}
Kenneth~R. Meyer.
\newblock Symmetries and integrals in mechanics.
\newblock In {\em Dynamical systems ({P}roc. {S}ympos., {U}niv. {B}ahia,
  {S}alvador, 1971)}, pages 259--272. Academic Press, New York, 1973.

\bibitem[Mos67]{Moser'67}
J{\"u}rgen Moser.
\newblock Convergent series expansions for quasi-periodic motions.
\newblock {\em Math. Ann.}, 169:136--176, 1967.

\bibitem[MW74]{MarsdenW74}
Jerrold Marsden and Alan Weinstein.
\newblock Reduction of symplectic manifolds with symmetry.
\newblock {\em Rep. Mathematical Phys.}, 5(1):121--130, 1974.

\bibitem[MW01]{MarsdenW01}
Jerrold~E. Marsden and Alan Weinstein.
\newblock Comments on the history, theory, and applications of symplectic
  reduction.
\newblock In {\em Quantization of singular symplectic quotients}, volume 198 of
  {\em Progr. Math.}, pages 1--19. Birkh\"auser, Basel, 2001.

\bibitem[PF92]{PasturF92}
Leonid Pastur and Alexander Figotin.
\newblock {\em Spectra of random and almost-periodic operators}, volume 297 of
  {\em Grundlehren der Mathematischen Wissenschaften [Fundamental Principles of
  Mathematical Sciences]}.
\newblock Springer-Verlag, Berlin, 1992.

\bibitem[SdlL12a]{SuL2}
Xifeng Su and Rafael de~la Llave.
\newblock K{AM} theory for quasi-periodic equilibria in 1{D} quasi-periodic
  media: {II}. {L}ong-range interactions.
\newblock {\em J. Phys. A}, 45(45):455203, 24, 2012.

\bibitem[SdlL12b]{SuL1}
Xifeng Su and Rafael de~la Llave.
\newblock K{AM} {T}heory for {Q}uasi-periodic {E}quilibria in
  {O}ne-{D}imensional {Q}uasi-periodic {M}edia.
\newblock {\em SIAM J. Math. Anal.}, 44(6):3901--3927, 2012.

\bibitem[Sel92]{Selke2}
W~Selke.
\newblock Spatially modulated structures in systems with competing
  interactions.
\newblock In {\em Phase transitions and critical phenomena, Volume 15}, pages
  1--72. Academic Press, 1992.

\bibitem[Sou97]{Souriau}
J.-M. Souriau.
\newblock {\em Structure of dynamical systems}, volume 149 of {\em Progress in
  Mathematics}.
\newblock Birkh\"auser Boston, Inc., Boston, MA, 1997.
\newblock A symplectic view of physics, Translated from the French by C. H.
  Cushman-de Vries, Translation edited and with a preface by R. H. Cushman and
  G. M. Tuynman.

\bibitem[vEF02]{vanErp'02}
T.~S. van Erp and A.~Fasolino.
\newblock {A}ubry transition studied by direct evaluation of the modulation
  functions of infinite incommensurate systems.
\newblock {\em Europhys. Lett.}, 59(3):330--336, 2002.

\bibitem[vEFJ01]{vanErp'01}
T.~S. van Erp, A.~Fasolino, and T.~Janssen.
\newblock Structural transitions and phonon localization in {F}renkel
  {K}ontorova models with quasi-periodic potentials.
\newblock {\em Ferroelectrics}, 250:421--424, 2001.

\bibitem[vEFRJ99]{vanErp'99}
T.~S. van Erp, A.~Fasolino, O.~Radulescu, and T.~Janssen.
\newblock Pinning and phonon localization in {F}renkel-{K}ontorova models on
  quasiperiodic substrates.
\newblock {\em Physical Review B}, 60(9):6522--6528, 1999.

\end{thebibliography}
\end{document}